\documentclass{amsart}
\usepackage{amsmath}
\usepackage{amsthm}
\usepackage{amssymb}

\usepackage{enumitem}
\usepackage{mathrsfs}

\usepackage[foot]{amsaddr}

\newtheorem{theorem}{Theorem}[section]

\newtheorem*{maintheorem*}{Main Theorem}

\theoremstyle{definition}

\newtheorem{proposition}[theorem]{Proposition}


\newtheorem*{proposition*}{Proposition}
\newtheorem*{observation*}{Observation}
\newtheorem*{claim*}{Claim}

\newtheorem*{lemma*}{Lemma}

\newtheorem{corollary}[theorem]{Corollary}

\newtheorem*{remark*}{Remark}

\newtheorem*{conjecture*}{Conjecture}

\newtheorem*{convention*}{Convention}

\theoremstyle{plain}
\newtheorem{lemma}[theorem]{Lemma}

\renewcommand{\tilde}{\widetilde}

\newcommand{\abs}[1]{\left|#1\right|}

\newcommand{\e}{\varepsilon}

\usepackage{mathrsfs}
\usepackage[mathscr]{euscript}

\newcommand{\NN}{\mathbb{N}}
\newcommand{\QQ}{\mathbb{Q}}

\newcommand{\RR}{\mathbb{R}}

\newcommand{\N}{\NN}

\newcommand{\cA}{\mathcal{A}}
\newcommand{\cB}{\mathcal{B}}

\newcommand{\cL}{\mathcal{L}}

\newcommand{\parbreak}[1]{
\begin{center}
***
\end{center}
}

\usepackage[dvipsnames]{xcolor}

\newcommand{\cl}[1]{\operatorname{cl}#1}

\usepackage{soul}

\usepackage{stmaryrd}
\newcommand{\ifbra}[1]{\left\llbracket #1 \right\rrbracket}

\begin{document}

\author[J. Byszewski]{Jakub Byszewski$^{2}$}
\email{jakub.byszewski@gmail.com}

\author[J.\ Konieczny]{Jakub Konieczny$^{1,2}$}
\email{jakub.konieczny@gmail.com}

\address[1]{Einstein Institute of Mathematics\\ Edmond J. Safra Campus\\ The Hebrew University of Jerusalem\\ Givat Ram\\ Jerusalem, 9190401\\ Israel}

\address[2]{Department of Mathematics and Computer Science\\
Institute of Mathematics\\
Jagiellonian University\\
ul. prof. Stanis\l{}awa \L{}ojasiewicza 6\\
30-348 Krak\'{o}w\\
Poland}

\title{A density version of Cobham's theorem}

\maketitle 

\begin{abstract}
Cobham's theorem asserts that if a sequence is automatic with respect to two multiplicatively independent bases, then it is ultimately periodic. We prove a stronger density version of the result: if two sequences which are automatic with respect to two multiplicatively independent bases coincide on a set of density one, then they also coincide on a set of density one with a periodic sequence. We apply the result to a problem of Deshouillers and Ruzsa concerning the least nonzero digit of $n!$ in base $12$. 
\end{abstract}

\section{Introduction}

A $k$-automatic sequence is a sequence whose $n$-th term is produced from the digits of $n$ in base $k$ using a finite procedure (more precisely, a deterministic finite automaton with output). The celebrated theorem of Cobham \cite{Cobham-1969} states that a sequence $a$ cannot be simultaneously automatic with respect to two different bases $k$ and $l$, except for the trivial cases when $k$ and $l$ are both powers of the same integer (in which case the notions of $k$- and $l$-automatic sequence coincide) or when $a$ is ultimately periodic (in which case it is automatic with respect to any base). The main aim of this paper is to weaken the hypothesis of Cobham's theorem by assuming that a sequence is $k$- and $l$-automatic almost everywhere.

Let $k$ and $l$ denote integers greater or equal than $2$. We say that $k$ and $l$ are \emph{multiplicatively independent} if they are not both powers of the same integer; equivalently, $\log k/ \log l \in \RR \setminus \QQ$. For a set of integers $A \subset \NN_0$, the \emph{upper density} of $A$  is given by the formula
$$\bar{d}(A) = \limsup_{n \to \infty} \frac{\abs{ A \cap [0,n) }}{n}.$$
We say that two sequences $a$ and $b$ coincide \emph{almost everywhere} (respectively, \emph{ultimately}) if the set of $n$ for which $a(n)\neq b(n)$ has upper density zero (respectively, is finite). A sequence is  \emph{periodic almost everywhere} if it coincides  almost everywhere with some periodic sequence. We define \emph{ultimately periodic} sequences analogously. 

Motivated by results in \cite{DeshouillersRuzsa-2011,Deshouillers-2012}, Deshouillers asked if a version of Cobham's theorem holds for \emph{almost everywhere automatic} sequences; see \cite[Sec.\ 3.5]{Allouche-2015} for details. Stated in a slightly different language, our main result answers this question in the positive.

\begin{maintheorem*} Let $k, l\geq 2$ be multiplicatively independent integers. Let $a$ be a $k$-automatic sequence and let $b$ be an $l$-automatic sequence. If $a$ and $b$ coincide almost everywhere, then they are both periodic  almost everywhere.
\end{maintheorem*}

Cobham's theorem has sparked a lot subsequent research.

A multidimensional version of Cobham's theorem was proved by Semenov \cite{Semenov-1977}, with alternative proofs by Durand in the language of symbolic dynamics \cite{Durand-2008} and Michaux and Villemaire in the language of logic \cite{MichauxVillemaire-1996}. Analogous results for other numeration systems are obtained also in \cite{ Fabre-1994, PointBruyere-1997, Durand-1998b, Hansel-1998,  Bes-2000, DurandRigo-2009}. For related results, see also \cite{Bes-1997}. Culminating a series of papers including \cite{ Fabre-1994, Durand-1998, Durand-2002-AA}, Durand \cite{Durand-2011} proved an analogue of Cobham's theorem for morphic sequences. This problem is closely related to the variant for different numeration systems.

An analogue for fractals was obtained by Adamczewski and Bell \cite{AdamczewskiBell-2011}, with a later generalisation by Chan and Hare \cite{ChanHare-2014}. This point of view has been further extended by Charlier, Leroy, and Rigo \cite{CharlierLeroyRigo-2015}. An analogue for quasi-automatic sequences (with applications to function fields in positive characteristic) was obtained by Adamczewski and Bell \cite{AdamczewskiBell-2008}. An analogue for regular sequences was obtained by Bell \cite{Bell-2005}. This work was further extended by Adamczewski and Bell to the case of Mahler series \cite{AdamczewskiBell-2013}, and another proof  was recently given by Sch\"{a}fke and Singer \cite{SchafkeSinger-2017}. An analogue for real numbers was obtained by Boigelot and Brusten \cite{BoigelotBrusten-2009}, see also \cite{ BoigelotBrustenLeroux-2009, BoigelotBrustenBruyere-2010}. Related results for Gaussian integers were obtained by Hansel and Safer \cite{HanselSafer-2003} and Bosma, Fokkink, and Krebs \cite{BosmaFokkinkKrebs-2017}.

Considerable effort has also gone into simplifying the proof of Cobham's theorem. Michaux and Villemaire \cite{MichauxVillemaire-1993} give a proof of Cobham's theorem using the language of logic. The original proof of Cobham has been simplified by Hansel \cite{Hansel-1982} (see also \cite{Reutenauer-1984}). A presentation of the proof is given in \cite{Perrin-1990} and \cite{AlloucheShallit-book}, and should be read together with \cite{Rigo-2006} which fills a gap in the proof noticed by K\"arki \cite{Karki-2005} and Rigo and Waxweiler. A recent presentation of the proof was also given by Shallit \cite{Shallit-2017}.
For an extended discussion of various variants of Cobham's theorem, we refer to the survey papers \cite{BruyereHanselMichauxVillemaire-1994, Durand-2002, DurandRigo-2011}.

We briefly describe the contents of the paper. In section 2, we introduce the basic concepts and state a few preliminary lemmas concerning automata. In section 3, we give a proof of the main result using combinatorial and automata-theoretic methods. Note that our proof relies on the classical version of Cobham's theorem. In the following section, we apply the main result to a problem of Deshouillers--Ruzsa concerning the least nonzero digit of the expansion of $n!$ in base $12$. Finally, in a short appendix we give an alternative proof of the main result under the extra assumption that both sequences are uniformly recurrent. In this case the result follows immediately from a theorem of Fagnot, but we also show how it can be derived from an arguably easier result of Allouche--Rampersad--Shallit which states that the lexicographically minimal element in the dynamical system generated by an automatic sequence is automatic.

\subsection*{Notation} We mainly use standard notation, with the possible exception of the Iverson bracket. If $\varphi$ is a sentence, then $\ifbra{\varphi} = 1$ if $\varphi$ is true, and $\ifbra{\varphi} = 0$ if $\varphi$ is false.

\subsection*{Acknowledgements}
The authors thank Jean-Paul Allouche, Jean-Marc Deshouillers, and Jeffrey Shallit for their useful suggestions. JK is also grateful to the University of Oxford, where part of this work was completed. JK  is supported by the European Research Council (ERC) under grant ``Ergodic Theory and Additive Combinatorics''.

\section{Preliminaries}

Let $\Omega$ be a finite set. We denote the set of all finite words over $\Omega$ by $\Omega^*$ and the set of all (right)-infinite words over $\Omega$ by $\Omega^{\omega}$. We identify $\Omega$-valued sequences  with elements of $\Omega^{\omega}$. A \emph{factor} of an infinite word $a\in\Omega^{\omega}$ is a finite word consisting of a number of consecutive symbols of $a$. For $a=a_0a_1a_2\cdots \in \Omega^{\omega}$ and integers $0\leq i\leq j$, we write $a_{[i,j)}$ for the factor $a_{[i,j)}=a_i a_{i+1} \cdots a_{j-1}$. The \emph{length} of a factor $u$ is the number of symbols it contains and is denoted by $|u|$. The set of all factors of $a$ is called the \emph{language} of $a$ and is denoted by $\cL(a)$. A prefix of a word $a$ is a factor of the form $a_{[0,j)}$ with $j\geq 0$. 

We endow the space $\Omega^{\omega}$ with the product topology, using discrete topology on the space $\Omega$. Let $\sigma\colon \Omega^{\omega} \to \Omega^{\omega}$ be the shift map given by $\sigma(a_0a_1a_2\cdots)=a_1a_2a_3\cdots$. For a word $a\in \Omega^{\omega}$, the \emph{subshift} $X_a$ induced by $a$ is the closure $X_a=\cl(\{ \sigma^n a \mid n \geq 0\})$ of the orbit of $a$. It is easy to see that $X_a$ consists exactly of the words $b\in \Omega^{\omega}$  such that each factor of $b$ is a factor of $a$. A word $a\in \Omega^{\omega}$ is called \emph{uniformly recurrent} if each factor $w$ of $a$ occurs in $a$ with bounded gaps, i.e., if there exists a constant $C$ (depending on $w$) such that every factor of $a$ of length at least $C$ contains $w$ as its factor. 

Let $k\geq 2$ be an integer. We work with $k$-automata (more precisely, deterministic finite $k$-automata with output) $\cA = (S,\Sigma_k,\delta,s_0,\Omega,\tau)$, where $S$ is the set of states, $\Sigma_k=\{0,1,\ldots,k-1\}$ is the input alphabet, $\delta$ is the transition function, $s_0$ is the initial state, $\Omega$ is the output alphabet, and $\tau$ is the output map. All the automata will read the input starting with the leading digit. For $n\in \N_0$, we denote by $(n)_k \in \Sigma_k^*$ the representation of $n$ in base $k$ without leading $0$'s.

 Let $\cA = (S, \Sigma_k,\delta, s_0,\Omega, \tau)$ be a $k$-automaton. We say that a state $s$ is \emph{accessible} if there exists a word $v\in \Sigma_k^*$ such that $\delta(s_0,v)=s$. Replacing $S$ with the set of accessible states if necessary, we can --- and will --- always assume that all states are accessible in all automata we consider.
 
 We say that $\cA$ \emph{ignores the leading $0$'s} if $\delta(s_0,0) = s_0$. We  say that $\cA$ is \emph{idempotent} if the action $\delta(\cdot,0)$ of $0$ on the set of states is idempotent, meaning that for each state $s \in S$ we have $\delta(s,00) = \delta(s,0)$. 

A \emph{strongly connected} component of a $k$-automaton $\cA = (S,\Sigma_k,\delta,  s_0,\Omega, \tau)$ is a subset of states $C \subset S$ such that $\delta(s,u) \in C$ for $s\in C$ and $u\in \Sigma_k$ and such that $C$ is strongly connected, i.e., for each $s,s' \in C$ there exists a word $w\in \Sigma_k^*$ such that $\delta(s,w)=s'$. For a state $s$, 
we denote by $a_{\cA,s}$ the $k$-automatic sequence induced by the automaton $(S,\Sigma_k,\delta,s,\Omega,\tau)$.

We begin with a few easy lemmas.

\begin{lemma}\label{lemma:densityconncomponent}
	Let $\cA = (S,\Sigma_k,\delta,  s_0,\Omega, \tau)$ be a  $k$-automaton. Then the set of $n\in \N_0$ such that $\delta(s_0,(n)_k)$ does not lie in a strongly connected component of $\cA$ has upper density $0$.
\end{lemma}
\begin{proof} It is easy to see that there exists a word $w\in \Sigma_k^*$ such that $\delta(s,w)$ lies in a strongly connected component of $\cA$ for all $s\in S$. Thus, in order that $n$ not belong to a strongly connected component of $\cA$, it is necessary that $(n)_k$ not contain $w$ as a factor. The set of such integers has upper density $0$.
\end{proof}

\begin{lemma}\label{lemma:idempotentautomaton}
	Let $a$ be a $k$-automatic sequence. Then there exists  a power $k'$ of $k$ such that $a$ is produced by an idempotent $k'$-automaton which  ignores the leading $0$'s. 
\end{lemma}
\begin{proof}
	It is  easy to see that any $k$-automatic sequence is produced by an automaton ignoring the leading $0$'s. Let $\cA =  (S,\Sigma_k,\delta,  s_0,\Omega,\tau)$ be such a $k$-automaton. For $t \in \NN$, write each $u\in \Sigma_{k^t}$ in the form $u=\sum_{i=0}^{t-1} u_i k^{i}$ with $u_i\in \Sigma_k$. Since $\cA$ ignores the leading $0$'s, the sequence $a$ is produced by a $k^t$-automaton $\cA_t = (S,\Sigma_{k^t},  \delta^{(t)},s_0, \Omega, \tau)$, where the transition function $\delta^{(t)}$ is given by $\delta^{(t)}(s,u) = \delta(s, u_{t-1}\cdots u_0)$. In particular, $\delta^{(t)}(\cdot,0)$ is the $t$-fold composition of $\delta(\cdot,0)$. Therefore for $k' = k^t$ with  $t \geq 1$ divisible by all the integers $\leq \abs{S}$, the automaton $\cA_t$ is idempotent and ignores the leading $0$'s.
\end{proof}

\begin{lemma}\label{lem:thick->component}
	Let $a \colon \NN_0 \to \Omega$ be a $k$-automatic sequence produced by an idempotent automaton $\cA = (S,s_0,\Sigma_k,\delta,\Omega,\tau)$ which ignores the leading $0$'s. Let $X_a$ be the subshift generated by $a$ and let $x \in \Omega$. We identify $x$ with the constant word $x^{\omega} \in \Omega^{\omega}$. Then the following conditions are equivalent: \begin{enumerate} 
	\item The point $x$ belongs to $X_a$. 
	\item There exists a strongly connected component $C$ of $\cA$ such that $\tau(s) = x$ for all $s \in C$.
	\end{enumerate}
\end{lemma}
\begin{proof}

Let us first assume that $x\in X_a$. It is easy to see that there exists a word $w\in \Sigma_k^*$ such that $\delta(s,w)$ lies in a  strongly connected component of $\cA$ for all states $s\in S$. We may further assume that $w$ has no leading $0$'s. For two states $s,s'$ in the same strongly connected component of $\cA$ there exists a word $v_{s,s'}\in \Sigma_k^*$ such that $\delta(s,v_{s,s'})=s'$. Let $t \geq 0$ be such that all the words $v_{s,s'}$ have length $\leq t$. 

Since $x$ belongs to $X_a$, there are factors of $a$ consisting of arbitrarily long sequences of the symbol $x$. Thus, there exists a word $u \in \Sigma_k^*$ with no leading $0$ such that $\delta(s_0,uw0v)=x$ for all $v\in \Sigma_k^t$.
By the construction of $w$, the state $s = \delta(s_0, uw0)$ lies in some  strongly connected component $C$ of $\cA$. For any $s' \in C$, we have a word $v_{s,s'} \in \Sigma_k^i$ with $i\leq t$ and $\delta(s,v_{s,s'}) = s'$. It follows (using the fact that $\cA$ is idempotent) that
	\begin{equation*}
		\tau(s') = \tau(\delta(s,v_{s,s'}))= \tau( \delta(s_0, uw0^{t-i+1}v_{s,s'})) = x.
	\end{equation*}

Now assume that there exists a  strongly connected component $C$ of $\cA$ such that $\tau(s)=x$ for all $s\in C$. Since $\cA$ ignores the leading $0$'s and since every state is accessible, there exists a word $w\in \Sigma_k^*$ such that $\delta(s_0,w)\in C$. Let $m$ be an integer such that $(m)_k=w$. By the assumption on C, we have $a(mk^t+i)=x$ for $t\geq 0$ and $0\leq i<k^t$, and hence $x^{\omega}$ lies in $X_a$. 
\end{proof}

\begin{remark*}
	The assumption of idempotence is essential in the above lemma. Indeed, for the sequence $a(n) = \abs{(n)_k} \bmod{2}$ that describes the parity of the number of digits of $n$ in base $k$, the subshift $X_a$ contains both the constant word $1^{\omega}$ and the constant word $0^{\omega}$, but $a$ is produced by a $k$-automaton $\cA$ which ignores the leading $0$'s and has precisely one strongly connected component.
\end{remark*} 
\section{Proof of the Main Theorem}

This section is devoted to the proof of the main result. The first step towards this goal is the following lemma.

\begin{lemma}\label{lem:positive density}
Let $k,l\geq 2$ be multiplicatively independent integers. Let  $a$ be a $k$-automatic sequence produced by a $k$-automaton $\cA$ and let $b$ be an $l$-automatic sequence  produced by an $l$-automaton $\cB$. Assume that $\cA$ and $\cB$ are idempotent and ignore the leading $0$'s. Let $C$ be a  strongly connected component of $\cA$ and let $D$ be a  strongly connected component of $\cB$. 

Then, there exist states $s\in C$ and $r\in D$ such that the automatic sequences $\tilde{a}=a_{\cA,s}$ and $\tilde{b}=a_{\cB,r}$ satisfy the following property: For any integer $m$ the set $$Z_m=\{ n \in \N_0 \mid a_{[n,n+m)} = \tilde{a}_{[0,m)} \text{ and } b_{[n,n+m)} = \tilde{b}_{[0,m)}\}$$ has positive upper density.

Assume further that the sequences $a$ and $b$ agree almost everywhere. Then the sequences $\tilde{a}$ and $\tilde{b}$ constructed above are equal.
\end{lemma}
\begin{proof}
	Write $\cA=(S,\Sigma_k,\delta,  s_0,\Omega,\tau)$ and $\cB=(S',\Sigma_k,\delta',  s'_0,\Omega',\tau')$. 	Pick some $u \in \Sigma_k^*$, $v \in \Sigma_l^*$ such that $\delta(s_0, u) \in C$ and $\delta'(s'_0,v) \in D$; we may additionally assume that $u$ and $v$ begin with nonzero digits. For any $t \geq 0$, we consider the sets 
	$$
		A^{(t)} = \{ n \in \NN_0 \mid (n)_k = u x 0^{t+1},\ (n)_l = v y 0^{t+1} \text{ for some } x \in \Sigma_k^*,\ y \in \Sigma_l^*\}. 
	$$
The interest in these sets stems from the observation that for  $m < \min(k^t,l^t)$ and $n \in A^{(t)}$ we have the equalities
$$
	a(n + m) = \tau(\delta(\delta(s_0, (n)_k), (m)_k)),\qquad 
	b(n + m) = \tau'(\delta'(\delta'(s'_0, (n)_l), (m)_l)),
$$
and moreover $\delta(s_0, (n)_k) \in C$ and $\delta'(s'_0, (n)_l ) \in D$. (Here, we use the fact that the automata are idempotent.) We may further split $A^{(t)}$ into a union of disjoint pieces
$$	
	A^{(t)}_{s,r} = \{ n \in A^{(t)} \mid \delta( s_0, (n)_k ) = s,\ \delta'( s'_0, (n)_l ) = r\}
$$
for $s \in C,\ r \in D$. 

We claim that the sets $A^{(t)}$ have positive upper density. First, note that there exist arbitrarily large integers $N$ such that $$(N)_k = u0x0^{t+1} \quad \text{and} \quad (N)_l = v0y0^{t+1}$$ for some $x \in \Sigma_k^*$ and $y \in \Sigma_l^*$. 
Indeed, since $k$ and $l$ are multiplicatively independent, the set $\{k^i/l^j \mid i,j\in \N_0\}$ is dense in the positive reals, and hence there exist arbitrarily large integers $i$ and $j$ such that the numbers $N = k^i l^j$ are of this form. (Pick $i \geq t+1$ such that $(k^{i})_l = v0y'$ and $j \geq t+1$ such that $(l^{j})_k = u0x'$ for some $x' \in \Sigma_k^*,\ y' \in \Sigma_l^*$.) There exists a constant $\e > 0$, independent of the choice of $N$, such that for any $N$ as above and for any $0 \leq m < \e N$ we have $N + m k^{t+1} l^{t+1} \in A^{(t)}$ (we may take $\e = 1/k^{\abs{u}+t+2}l^{\abs{v}+t+2}$). Counting the number of choices of $m$, we conclude that $$\bar{d}(A^{(t)}) \geq \e/(1+\e k^{t+1}l^{t+1}) > 0.$$
	
Since the upper density is subadditive, for each $t \geq 0$ there exist $s \in C$ and $r \in D$ such that $\bar{d}(A^{(t)}_{s,r}) > 0$. Since $A^{(t+1)}_{s,r} \subset A^{(t)}_{s,r}$, we may find $s \in C$ and $r \in D$ such that $\bar{d}(A^{(t)}_{s,r}) > 0$ for all $t \geq 0$.

Fix such a choice of $s$ and $r$, and let the sequences $\tilde a(m) = a_{\cA,s}(m) = \tau( \delta(s, (m)_k))$ and $\tilde b(m) = a_{\cB,r} = \tau'( \delta'( r, (m)_l))$ be defined as above. Put $M(t) = \min(k^t,l^t)$. By construction of the sets $A^{(t)}_{s,r}$, for $n \in  A^{(t)}_{s,r}$ and $0 \leq m < M(t)$ we have
\begin{equation*}
	a(n + m) = \tilde a(m), \qquad b(n + m) = \tilde b(m). \label{eq:201}
\end{equation*}

This proves that $A^{(t)}_{s,r} \subset Z_{M(t)}$, and hence $Z_{M(t)}$ has positive density. Since the sequence of sets $Z_m$ is descending, letting $t \to \infty$ concludes the proof of  the first part of the claim. 

In order to prove  the remaining part of the statement, suppose that $\tilde a(m) \neq \tilde b(m)$ for some $m\geq 0$. Fix an integer $t \geq 0$ with $m < M(t)$ and note that for $n\in A^{(t)}_{s,r}$ we have $$a(n+m)=\tilde{a}(m)\neq\tilde{b}(m)=b(n+m).$$ Since $A^{(t)}_{s,r}$ has positive density, we reach a contradiction with the assumption that $a$ and $b$ are equal almost everywhere.
\end{proof}

We are now ready to prove our main result.

\begin{proof}[Proof of the Main Theorem] Let $k, l\geq 2$ be multiplicatively independent integers. Let $a$ be a $k$-automatic sequence and let $b$ be an $l$-automatic sequence that coincide almost everywhere.

 By Lemma \ref{lemma:idempotentautomaton}, after replacing $k$ and $l$ by their appropriate powers, there exist automata $\cA$ and $\cB$ that are idempotent, ignore the leading $0$'s, and produce the sequences $a$ and $b$, respectively. Applying Lemma \ref{lem:positive density} to $\cA$ and $\cB$, we obtain the sequence $\tilde{a}=\tilde{b}$ (see Lemma \ref{lem:positive density} for notation). Since the sequence $\tilde{a}=\tilde{b}$ is both $k$- and $l$-automatic, the classical version of Cobham's theorem shows that it is ultimately periodic.
 
 Let $c$ be a periodic sequence such that $\tilde{a}=\tilde{b}$ and $c$ ultimately coincide and let $q$ be the period of $c$.
	Ignoring the first finitely many terms where $\tilde{a}(m)\neq c(m)$ and splitting the set $Z_m$ into finitely many pieces, we conclude that the set
	 $$Z_m^{(i)}=\{n\in \N_0\mid a_{[n,n+m)}=c_{[0,m)} \text{ and } n \equiv i \bmod q\}$$
	 is infinite for some $0\leq i \leq q-1$ and all $m\in \N_0$. Thus, after possibly replacing the periodic sequence $c$ by its shift $\tilde{c}$, we see that $a$ and $\tilde{c}$ have arbitrarily long common factors. Let $a'$ and $b'$ be sequences given by $a'(n)=\ifbra{a(n)=\tilde{c}(n)}$ and $b'(n)=\ifbra{b(n)=\tilde{c}(n)}$, where we are using the Iverson bracket notation.
	 
	By construction, the sequences $a'$ and $b'$ take only the values $0$ and $1$,  $a'$ is $k$-automatic and $b'$ is $l$-automatic. It is also clear that $a' $ and $b'$ coincide almost everywhere. Moreover, since $Z_m^{(i)}$ is infinite, the constant word $1^{\omega}$ belongs to the subshift $X_{a'}$ generated by $a'$. Applying  Lemma  \ref{lemma:idempotentautomaton} again, we obtain automata $\cA'$ and $\cB'$ that are idempotent, ignore the leading $0$'s and produce the sequences $a'$ and $b'$, respectively.
	
	By Lemma \ref{lem:thick->component}, $\cA'$ has a strongly connected component where the output function takes the constant value $1$. Applying Lemma \ref{lem:positive density} to all strongly connected components of $\cA'$ and $\cB'$, we conclude that in both of these automata, the output function takes the value $1$ on every state in any strongly connected component. Thus, by Lemma \ref{lemma:densityconncomponent} the sequences $a'$ and $b'$ are equal to $1$ almost everywhere, and hence $a=\tilde{c}=b$ almost everywhere.
\end{proof} 
\section{An application}

In \cite{DeshouillersRuzsa-2011, Deshouillers-2012, Deshouillers-2016}, Deshouillers and Ruzsa studied the sequence $t_{12}(n!)$ of least nonzero digits of $n!$ in base $12$. More generally, one can consider the sequence $t_k(n!)$ of least nonzero digits of $n!$ in base $k$ for any $k\geq 2$ (see \cite{Kakutani-1967, Dekking-1980, Dresden-2008}).

	If $k$ is a prime power, then Deshouillers \cite{Deshouillers-2012} notes that $t_k(n!)$ is $k$-automatic. More generally, if $k$ is written in the form
$$
	k = p_1^{\alpha_1} p_2^{\alpha_2}  p_3^{\alpha_3}  \dots,
$$
where $p_i$ are distinct primes arranged so that $$ \alpha_1(p_1 -1) \geq\alpha_2 (p_2 -1) \geq\alpha_3 (p_3 -1) \geq \dots ,$$
then $t_k(n!)$ is $p_1$-automatic under the weaker assumption that the first inequality above is strict: $\alpha_1 (p_1 -1) > \alpha_2 (p_2 -1)$. In particular, the sequence $t_{10}(n!)$ is $5$-automatic. (This last conclusion was first proved in \cite{Kakutani-1967}.)

The base $k = 12$ is the first instance when the strict inequality  $\alpha_1 (p_1 -1) > \alpha_2 (p_2 -1)$ does not hold. Hence, it is natural to ask if $t_{12}(n!)$ is automatic, and more precisely if the sequences  $\ifbra{t_{12}(n!) = y}$ are automatic for $0 \leq y < 12$. Partial progress towards this goal has been made by Deshouillers and Ruzsa. In \cite{DeshouillersRuzsa-2011}, it is shown that $t_{12}(n!)$ coincides almost everywhere with a $3$-automatic sequence which takes only the values $4$ and $8$. On the other hand, as shown in \cite{Deshouillers-2012}, $t_{12}(n!)$ takes each of the values $3,6,9$ infinitely often. Moreover, the same author \cite{Deshouillers-2016} proved that the sequence $\ifbra{t_{12}(n!) = y}$ is not automatic for $y = 3,6,9$, and is not $3$-automatic for $y = 4,8$. It is natural to ask whether for $y = 4,8$, the sequence $\ifbra{t_{12}(n!) = y}$ might be automatic in a different base. Using the density version of Cobham's theorem, we are able to answer this question.

\begin{corollary}
	For any $y \in\{ 3, 4, 6, 8, 9\}$, the sequence $\ifbra{t_{12}(n!) = y}$ is not automatic. 
\end{corollary}
\begin{proof}
	The cases where $y \in \{3, 6 ,9\}$ have already  been treated in \cite[Theorem 1]{Deshouillers-2016}. Thus, let $y \in \{4,8\}$, and suppose for the sake of contradiction that the sequence $a(n)$ is $k$-automatic for some $k \geq 2$. By \cite{DeshouillersRuzsa-2011}, there exists a $3$-automatic sequence $b(n)$ such that $a(n) = b(n)$ almost everywhere. Moreover, it follows from the construction in \cite{DeshouillersRuzsa-2011} (see also \cite[Proposition 1]{Deshouillers-2016}) that the value of $b(n)$ depends only on the parity of the number of digits of $(n)_9$ in $\{2,3,4,6,7\}$; $b(n) = 4$ if the digits $2,3,4,6,7$ appear among the base-$9$ digits of $n$ in total an even number of times, and $b(n) = 8$ otherwise. In particular, $b(n)$ is not almost everywhere periodic. Thus the Main Theorem implies that $k$ is a power of $3$, and hence $a(n)$ is $3$-automatic. However, this possibility has been already ruled out by Deshouillers \cite[Theorem 1]{Deshouillers-2016}.
\end{proof}
 
\appendix

\section{Uniformly recurrent sequences}

In the case when the sequences $a$ and $b$ are uniformly recurrent, the proof of the Main Theorem is particularly simple, and arguably more elegant. Indeed, it is a consequence of a different variant of Cobham's theorem, due to Fagnot.

\begin{theorem}[Fagnot]\label{thm:Fagnot}
	Let $k, l\geq 2$ be multiplicatively independent integers. Let $a$ be a $k$-automatic sequence and let $b$ be an $l$-automatic sequence. Suppose that $\cL(a) = \cL(b)$. Then $a$ and $b$ are ultimately periodic.	
\end{theorem}
\begin{proof}
\cite[Th\'eor\`eme 15]{Fagnot-1997}.
\end{proof}

\begin{proposition}\label{prop:main-SC}  Let $k, l\geq 2$ be multiplicatively independent integers. Let $a$ be a $k$-automatic sequence and let $b$ be an $l$-automatic sequence. Suppose that $a$ and $b$ coincide almost everywhere and are both uniformly recurrent. Then $a$ and $b$ are equal and periodic.
\end{proposition}
\begin{proof}
	Note that $\cL(a) = \cL(b)$; indeed, it follows from uniform recurrence that each factor appears in $a$ with positive frequency, and it follows from the fact that $a$ and $b$ coincide almost everywhere that each factor that appears in $a$ with positive frequency is also a factor of $b$. Hence, $\cL(a) \subset \cL(b)$. By symmetry, we obtain the opposite inclusion.

	It follows from Theorem \ref{thm:Fagnot} that $a$ and $b$ are ultimately periodic. Since a uniformly recurrent ultimately periodic sequence is periodic, and since two periodic sequences that agree almost everywhere are equal, we conclude that $a=b$, and that this common sequence is indeed periodic.
\end{proof}

Note that in the above proof we do not need the full strength of Theorem \ref{thm:Fagnot}. Indeed, the following weaker version suffices.

\begin{proposition}\label{prop:weak-Fagnot}
	Let $k, l\geq 2$ be multiplicatively independent integers. Let $a$ be a $k$-automatic sequence and let $b$ be an $l$-automatic sequence. Suppose that $\cL(a) = \cL(b)$, and $a$ and $b$ are uniformly recurrent. Then $a$ and $b$ are periodic.	
\end{proposition}

We give an independent proof of Proposition \ref{prop:weak-Fagnot}. For this purpose, we need some preliminaries.
Assume that $\Omega$ is further equipped with a total order $<$ and let $a \in \Omega^{\omega}$ be an infinite word. Consider the element $\min(a) \in \Omega^{\omega}$   characterised by the property that for each $n \geq 0$ the $n$-th prefix $\min(a)_{[0,n)}$ of $\min(a)$ is the lexicographically least factor of $a$ of length $n$. It is easy to see that the word $\min(a)$ is the smallest element of the closed orbit $X_a$ of $a$ with respect to the lexicographic ordering on $\Omega^{\omega}$. We will use the following result.

\begin{theorem}[Allouche--Rampersad--Shallit]\label{thm:min-is-auto}
	Let $\Omega$ be a finite ordered set, and let $a \colon \NN_0 \to \Omega$ be a $k$-automatic sequence. Then $\min(a)$ is $k$-automatic.
\end{theorem}
\begin{proof} \cite[Theorem 6]{AlloucheRampersadShallit-2009}.\end{proof}

\begin{proof}[Proof of Proposition \ref{prop:weak-Fagnot}]
	By an argument similar to the one in the proof of Proposition \ref{prop:main-SC}, one can show that $\cL(a) = \cL( \min(a))$ and analogously $\cL(b) = \cL(\min(b))$. Indeed, the inclusion $\cL(\min(a)) \subset \cL(a)$ is obvious. For the opposite inclusion, note that by uniform recurrence for each $n \geq 0$ there exists $m \geq 0$ such that each factor of $a$ of length $\leq n$ appears in each factor of $a$ of length $\geq m$. 

	By construction,  $\min(a)$ is determined by $\cL(a)$ and $\min(b)$ is determined by $\cL(b)$. Thus, $\cL(a) = \cL(b)$ implies that $\min(a) = \min(b)$; denote this sequence by $c$. By Theorem \ref{thm:min-is-auto}, $c$ is both $k$- and $l$-automatic. Hence, by the classical version of Cobham's theorem, $c$ is ultimately periodic. It follows that there is an integer $C$ such that for any $n$ the number of factors of $c$ of  length $n$ is at most $C$. Since $\cL(c)=\cL(a)=\cL(b)$, the same holds for $a$ and $b$, and hence by a well-known result, $a$ and $b$ are both ultimately periodic (see, e.g., \cite[Theorem 10.2.6]{AlloucheShallit-book}). 
\end{proof}
 
\bibliographystyle{amsalpha}
\bibliography{bibliography}

\end{document}